\documentclass[10pt,leqno]{amsart}

\usepackage{amsmath, amssymb, amsthm, stmaryrd,tikz,marginnote}
\usepackage{euscript}
\usepackage{verbatim}
\usepackage{lineno}
\usepackage{graphicx}
\usepackage{amsmath, verbatim}
\usepackage{amssymb,amsfonts,mathrsfs}
\usepackage[colorlinks=true, linkcolor=blue,citecolor=blue, bookmarks=true, pdfstartview=FitH]{hyperref}
\usepackage[driver=pdftex,heightrounded=true,centering]{geometry}
\usepackage{marginnote}
\usepackage{amsthm}

\newcommand{\vep}{\varepsilon}

\newcommand{\holder}{\mathcal{C}}

\DeclareMathOperator{\Rm}{Rm}
\DeclareMathOperator{\Rc}{Ric}

\newcommand{\gtil}{\widetilde{g}}

\newcommand{\nabtil}{\widetilde{\nabla}}

\renewcommand{\Re}{\mathrm{Re} \;}

\newcommand{\del}{\partial}

\newtheorem{theorem}{Theorem}[section]

\newtheorem{prop}[theorem]{Proposition}

\newtheorem{remark}[theorem]{Remark}
\newtheorem{corollary}[theorem]{Corollary}

\numberwithin{equation}{section}

\begin{document}
\title[Convergence Stability]{Convergence stability for Ricci flow \\ on manifolds with bounded geometry}
\author[Bahuaud]{Eric Bahuaud}
\address{Department of Mathematics,
	Seattle University,
	Seattle, WA, 98122, USA}
\email{bahuaude [AT] seattleu [DOT] edu}
\author[Guenther]{Christine Guenther}
\address{Department of Mathematics,
	Pacific University,
	Forest Grove, OR 97116, USA}
\email{guenther@pacificu.edu}
\author[Isenberg]{James Isenberg}
\address{Department of Mathematics,
	University of Oregon,
	Eugene, OR 97403-1222 USA}
\email{isenberg@uoregon.edu}
\author[Mazzeo]{Rafe Mazzeo}
\address{Department of Mathematics, Stanford University, Stanford, CA}
\email{rmazzeo@stanford.edu}

\date{\today}
\maketitle

\begin{abstract} We prove that the Ricci flow for complete metrics with bounded geometry depends continuously on initial conditions for finite time 
with no loss of regularity. This relies on our recent work where sectoriality for the generator of the Ricci-DeTurck flow is proved. We use this to 
prove that for initial metrics sufficiently close in H\"older norm to a rotationally symmetric asymptotically hyperbolic metric and satisfying a simple
curvature condition, but a priori distant from the hyperbolic metric, Ricci flow converges to the hyperbolic metric. 
\end{abstract} 

\section{Introduction}
This note focuses on two topics in the study of geometric flows.  We discuss only Ricci flow here but analogues of these results and 
observations apply more broadly.  The first topic is a general one concerning finite-time continuous dependence of solutions on initial 
conditions with no loss of regularity for chosen initial conditions. This is a technical point which is typically taken for granted, and slightly 
weaker versions appear explicitly elsewhere, cf.\ \cite{BGI} for one version in the compact setting.  The core issue is that while
the well-posedness on a particular function space of the Ricci flow with a particular choice of gauge (e.g., the DeTurck or Bianchi gauge) 
is more easily accessible, the `adjustment' needed to obtain a solution to the ungauged Ricci flow introduces some extra steps
for which one must check that time of existence and precise regularity are not lost. We state and prove a sharp 
and presumably optimal version of this well-posedness which holds in the setting of complete metrics with bounded geometry (see \cite[Section 2]{BGIM} for a review of this geometric setting). 

The second topic is a new stability result about Ricci flow of complete, asymptotically hyperbolic (AH) metrics.  The starting point is a 
result by the first author and Woolgar \cite{BW}, which states that if $g_0$ is AH and rotationally symmetric, and satisfies a mild geometric
hypothesis to be described below, then the Ricci flow solution $g(t)$ starting at $g_0$ converges to the standard hyperbolic metric. This is 
not perturbative since $g_0$ may be far away from the hyperbolic metric, and in particular might have some positive 
sectional curvatures. The new result here asserts the stability of the entire trajectory of such solutions, in the sense that if $g_1$ is any 
not necessarily rotationally symmetric AH metric sufficiently close to a rotationally symmetric metric $g_0$ satisfying the condition of \cite{BW}, 
then the Ricci flow solution $g_1(t)$ starting at $g_1$ also converges to the standard hyperbolic metric.   This is an instance of
convergence stability, as introduced in \cite{BGI}.  

This second result leads to a more general theorem about convergence stability around any given converging flow trajectory in other geometric
settings. This applies not only to Ricci flow but to more general geometric flows as well. We hope to develop this further in the future. 

The proofs of these results rely on the {\it sectoriality} of the linearized gauged Ricci flow operator acting on weighted (little) H\"older spaces.  To define sectoriality on a Banach space $X$, consider a closed unbounded linear operator $L$ with domain $D \subset X$. The \textit{resolvent set} is the set of $\lambda \in \mathbb{C}$ for which $(\lambda I - L): D \to X$ has bounded inverse, $(\lambda I - L)^{-1}$, which we call the \textit{resolvent operator}.  $L$ is sectorial on $X$ if there exists $\omega \in \mathbb{R}$ and $\theta \in (0,\frac{\pi}{2})$ such that the resolvent set contains a
sector of the form $S = \{ \lambda \in \mathbb{C} \setminus \{\omega\}: | \arg(\omega - \lambda)| < \theta \}$.  Such a sector contains the ``left'' half-plane $\Re (\lambda) < \omega$.  We further require that there is a constant $C>0$ so that for all $\lambda $ in this sector, the operator norm of the resolvent satisfies
\begin{equation*}
\| (\lambda I - L)^{-1} \| \leq \frac{C}{|\lambda - \omega|}.
\end{equation*}

As explained carefully in \cite{Lunardi}, this notion is a key technical ingredient
which makes it possible to apply analytic semigroup theory to obtain well-posedness for parabolic flows. The paper \cite{BGIM} proves 
sectoriality for a broad class of `geometric' elliptic operators on arbitrary complete manifolds of bounded geometry.  We prove that the Ricci 
flow for metrics with bounded geometry is well-posed, with finite time continuous dependence, in Section 2. The proof relies on well-posedness 
of the Ricci-DeTurck flow. In Section 3 we review the Ricci flow convergence result for asymptotically hyperbolic rotationally symmetric metrics 
from \cite{BW} and based on this, prove convergence stability for Ricci flow near any such trajectory. 

\

This version clarifies one argument in our published article appearing in the Proceedings of the American Mathematical Society.  There are small modifications to the statements of several Theorems.  We are grateful to Anuk Dayaprema for bringing this to our attention. 

\subsubsection*{Acknowledgments} This work was supported by collaboration grants from the Simons Foundation (\#426628, E. Bahuaud 
and \#283083, C. Guenther).  J. Isenberg was supported by NSF grant PHY-1707427. 

\section{Finite time well-posedness for Ricci flow}
Our first result is a sharp  quantitative version of well-posedness for the ungauged Ricci flow on H\"older spaces in bounded geometry.

\begin{theorem}[Well-posedness of Ricci Flow in H\"older spaces]
\label{continuousdependence}  \label{thm:ctsdep}
Let $(M,g_0)$ be a complete metric of bounded geometry with maximally defined Ricci flow $g_0(t)$ with $g_0(0) = g_0$ for $0 \leq t < T_{max}(g_0) \leq \infty$.  
For any $\tau \in (0, T_{max}(g_0))$, and $k \geq 2$, there exist $C_{\mathrm{init}}, C_{\mathrm{fin}} > 0$, depending only on $g_0$ and $\tau$, 
such that if 
\[
||g_1 - g_0||_{\holder^{k,\alpha}} \le C_{\mathrm{init}},
\]
then the unique bounded curvature Ricci flow solution $g_1(t)$ with $g_1(0) = g_1$ has maximal existence time 
$T_{max}(g_1) \ge \tau$, and 
\begin{equation*}
||g_1(t) - g_0(t) ||_{\holder^{k,\alpha}}  \le C_{\mathrm{fin}} || g_1 - g_0 ||_{\holder^{k,\alpha}} 
\end{equation*} 
for all $t \in [\tau/2,\tau].$
\end{theorem}
Note that the lower bound on time, $\tau/2$, appearing in the statement above may be replaced by any positive value less than $\tau$, which leads to the same H\"older estimate for a new constant $C_{\mathrm{fin}}$.
The method and proof here complete and extend \cite{BGI}, which contains a version this result, but only on compact manifolds.
This theorem  in a suitably modified form applies to other geometric flows, e.g., the obstruction flow as described in \cite{BGIM}. 

\medskip

The proof of Theorem \ref{thm:ctsdep} has two steps:
\begin{enumerate}
\item Proposition \ref{prop:ctsdep-rfrdf} (see below) shows that if the Ricci flow $g_0(t)$ exists on $[0,T']$, and there is an associated Ricci-DeTurck flow 
$\mathfrak{g}_0(t)$ on the same interval $[0,T']$ (for some choice of reference metric $\gtil$), then both of these solutions depend continuously (in $\mathcal C^{k,\alpha}$)
on the initial data in some small neighborhood of $g_0$. We also obtain a quantitative estimate of how far these solutions may drift from
one another up to time $T'$. 
\item Given a solution $g_0(t)$ to Ricci flow on $[0,T]$, we subdivide $[0,T]$ into finitely many subintervals $[t_{i-1},t_{i}]$
so that the harmonic map heat flow coupled to the Ricci flow yields a solution to \eqref{eqn:RDTF} on each of these intervals with reference metric 
$g(t_{i-1})$. Together with the first step, this proves the result.
\end{enumerate} 

\subsection{A brief review of Ricci-DeTurck flow}
This material is standard and explained carefully in many places, e.g., \cite{CLN}. 

Assume that $(M^n, g_0)$ is a complete manifold with bounded geometry.
As such, the curvature of $g_0$ is uniformly bounded. The Ricci flow is determined by the initial value problem 
\begin{equation} \label{eqn:RF}
\frac{\partial\, }{\partial t} g(t) = - 2 \Rc (g(t) ), \; \; g(0) = g_0.
\end{equation} 
As is well known, diffeomorphism invariance prevents this from being parabolic, but parabolicity is recovered by introducing a gauge. We use the DeTurck approach. 

Fix a reference metric $\gtil$ on $M$ which is also of bounded geometry and quasi-isometric to $g_0$ (i.e., for some positive constants $c_1$ and $c_2$, $c_1 \gtil \leq g_0 \leq c_2 \gtil$ and 
the $g_0$ covariant derivatives of $\gtil$ up to some order are bounded). Using any local coordinate system, define the 
components of the time-dependent vector field $W$ in terms of the Christoffel symbols of $\gtil$ and $g_0(t)$ by 
\begin{equation}
W^k(g(t)) = g(t)^{pq} \left( \Gamma(g(t))^{k}_{pq} - \Gamma(\gtil)_{pq}^k \right).
\end{equation}
Despite this coordinate-based definition, the vector field $W$ is well-defined. The Ricci-DeTurck flow with reference metric $\gtil$ is then defined by
\begin{equation} \label{eqn:RDTF}
\frac{\partial\, }{\partial t} \mathfrak{g}(t) = - 2 \Rc (\mathfrak{g}(t) ) + \mathcal L_{W(\mathfrak{g})} \mathfrak{g}(t), \; \; \mathfrak{g}(0) = \mathfrak{g}_0,
\end{equation} 
where $\mathcal L$ is the Lie derivative. This is a parabolic equation, and standard theory yields existence of a solution on a short time interval $[0, \vep_0)$. 

To pass from the solution of this gauged flow to the solution of \eqref{eqn:RF}, let $\phi_t$ be the family of diffeomorphisms generated by $-W$, i.e.,
\begin{equation} \label{eqn:ODEW}
\frac{d}{dt} \phi_t  = - W \circ \phi_t, \; \; \phi_0 = \mathrm{id}.
\end{equation} 
 $\phi_t$ also exists on some short time interval.
Then $g(t) = (\phi^{-1}) ^* \mathfrak{g}(t)$ solves \eqref{eqn:RF} on the time interval where solutions to both \eqref{eqn:RDTF} and \eqref{eqn:ODEW} exist.

This process -- proceeding from a solution of the Ricci-DeTurck flow to a solution of the Ricci flow -- can be reversed. To set this up as a parabolic equation, we couple Ricci flow to the harmonic map heat flow. 
Namely, suppose that $g(t)$ solves \eqref{eqn:RF} on $[0, T_{\max})$. Fix a reference metric $\gtil$ (again complete, bounded geometry and quasi-isometric to $g(0)$). 
Let $\Delta_{g(t),\gtil}$ denote the ($t$-dependent) map Laplacian associated to this pair of metrics (see \cite{AH} for a discussion of the map Laplacian in this setting). We seek a family of maps $f(t): (M, g(t)) \to (M, \gtil)$ 
solving the harmonic map heat flow 
\begin{equation} \label{eqn:HMHF}
\frac{\partial\, }{\partial t} f(t) = \Delta_{g_{0}(t),\gtil} \; f(t), \; \; f(0) = \mathrm{id}.
\end{equation} 
Then, so long as $f(t)$ remains a diffeomorphism, $\mathfrak{g}(t) =(f^{-1*}) g_0(t)$ satisfies \eqref{eqn:RDTF} with reference metric $\gtil$ 
and initial metric $g_0$.  

Solutions to \eqref{eqn:HMHF} exist and remain diffeomorphisms on some time interval $[0,T_1)$, but it is conceivable that $T_1 < T_{\max}$.
A result by Bamler and Kleiner \cite{BamlerKleiner} will show that this is not the case. 

\subsection{Well-posedness of Ricci-DeTurck flow}
\label{sec:ctsdep-rfrdtf}
We first prove continuous dependence of Ricci flow on initial conditions on any finite time interval, {\it assuming} existence of solutions of the 
associated Ricci-DeTurck flow on that same time interval. The point is simply to combine the well-posedness of the Ricci-DeTurck flow and
the flow determining the family of connecting diffeomorphisms. However, doing this naively leads to a drop in regularity, but an extra step
yields well-posedness with no such loss.  The `chaining' argument in the second step discussed above relies on this. 

\begin{prop}(Well-posedness of Ricci Flow contingent on Ricci-DeTurck flow).  \label{prop:ctsdep-rfrdf} Let $(M,g_0)$ be a complete Riemannian
manifold with bounded geometry. Suppose that $g_0(t)$ is a solution to Ricci flow on some interval $[0,T]$, $0 < T < T_{max}(g_0)$, and 
that relative to the background metric $\gtil = g_0$, the solution to the Ricci -DeTurck flow $\mathfrak{g}(t)$, $\mathfrak{g}(0) = g_0$ exists on the
same time interval. Then for any $k \geq 2$ and $\tau \in (0,T)$, there exist constants $C_{\mathrm{init}}, C_{\mathrm{fin}} > 0$, depending only on $k$, $g_0$ 
and $\tau$, such that if 
\[
||g_1 - g_0||_{\holder^{k,\alpha}} \le C_{\mathrm{init}},
\]
then the unique Ricci flow solution $g_1(t)$ with $g_1(0) = g_1$ exists for $0 \leq t \leq \tau$ and satisfies
\begin{equation}
||g_1(t) - g_0(t) ||_{\holder^{k-2,\alpha}}  \le C_{\mathrm{fin}} || g_1 - g_0 ||_{\holder^{k,\alpha}},\ \ 0 \leq t \leq \tau.
\label{stabest0}
\end{equation} 
Moreover, there exists a constant $C_{\mathrm{fin}}' > 0$ depending on $k$, $g_0$ and $\tau$, such that such that higher-derivatives satisfy
\begin{equation}
||g_1(t) - g_0(t) ||_{\holder^{k,\alpha}}  \le C_{\mathrm{fin}}' || g_1 - g_0 ||_{\holder^{k,\alpha}},\ \ \frac{\tau}{2} \leq t \leq \tau.
\label{stabest}
\end{equation} 
\end{prop}
\begin{proof}
A version of this theorem appears in \cite{BGI}, but the statement here is slightly sharper and more general, inasmuch as we consider 
flows on complete manifolds of bounded geometry, and in addition, the regularity on the two sides of \eqref{stabest} is the same,
whereas in \cite{BGI} there is a drop of two orders of regularity. We refer to \cite{BGI} for most of the details and computations. 

Denote the Ricci-DeTurck operator by $F^{\gtil}$, so that \eqref{eqn:RDTF} can be written simply as 
\begin{equation}
\partial_t \mathfrak{g}(t) = F^{\gtil}( \mathfrak{g}(t) ), \, \mathfrak{g}(0) = g_0.
\end{equation}
Clearly, $F^{\gtil}: \holder^{k+2,\alpha}(M) \longrightarrow \holder^{k,\alpha}(M)$ is a smooth map for any $k \geq 0$, with 
elliptic linearization $DF^{\gtil}_g$.  We now invoke \cite[Theorem 1]{BGIM}, which asserts that $DF^{\gtil}_g$ is sectorial 
on $\holder^{k,\alpha}(M)$.  Using this as in \cite[Theorem 6]{BGI}, we obtain constants $C^{\mathrm{DT}}_{\mathrm{init}}, C^{\mathrm{DT}}_{\mathrm{fin}} > 0$ 
so that if 
\[ 
\| g_1 - g_0 \|_{\holder^{k,\alpha}} < C^{\mathrm{DT}}_{\mathrm{init}}
\] 
then the solution $\mathfrak{g}_1(t)$ to the Ricci-DeTurck flow with $\gtil = g_0$ exists at least for $0 \leq t \leq \tau$ and satisfies
\begin{equation} \label{eqn:DTest}
\| \mathfrak{g}_1(t) - \mathfrak{g}_0(t) \|_{\holder^{k,\alpha}} < C^{\mathrm{DT}}_{\mathrm{fin}} \| g_1 - g_0 \|_{\holder^{k,\alpha}}
\end{equation} 
on that time interval. This is just the well-posedness of the Ricci-DeTurck flow. 

Following \cite[Theorem A]{BGI}, we now estimate the DeTurck vector fields \[W_i = W_i(\mathfrak{g}^{-1}_i(t), \Gamma(\mathfrak{g}_i(t)), \Gamma(g_0)),\] and their associated diffeomorphisms.   First assume that $k \geq 4$.  Using Shi's well-known estimates for the curvature tensor of the metrics in this
family \cite{Shi}, we can estimate $W_0$ entirely in terms of $\tau$ and the curvature bound for $g_0(t)$ on $[0,\tau]$.  On the other hand, using interpolating 
differences and \eqref{eqn:DTest} below,  we obtain that the norms of $W_1$ and $W_1 - W_0$ are controlled by $\tau$, the curvature of $g_0(t)$ and
the constant  $C^{\mathrm{DT}}_{\mathrm{fin}}$.  Integrating yields estimates for the diffeomorphisms $\phi_t$. We thus obtain a constant $C_{\mathrm{fin}} > 0$ 
such that the Ricci flow solutions $g_i = \phi_i^* \mathfrak{g}_i(t)$ satisfy
\[ 
\|g_1(t) - g_0(t) \|_{ \holder^{k-2,\alpha} } \leq C_{\mathrm{fin}} \| g_1 - g_0 \|_{\holder^{k,\alpha}}. 
\]
The loss of regularity occurs because the $W_i$ involve one derivative of the metric, and pull-back by a diffeomorphism involves a further derivative.

To obtain a stability estimate with no loss of regularity, we proceed as follows. First, by \cite[(7.44)]{CLN}, the difference $\mathfrak{g}_1(t) 
- \mathfrak{g}_0(t)$ satisfies the linear homogeneous parabolic equation
 \begin{align}
\label{diffmetrics}
\partial_t(\mathfrak{g}_1 - \mathfrak{g}_0) &= \mathfrak{g}_1^{jl} \nabtil_j \nabtil_l(\mathfrak{g}_1 - \mathfrak{g}_0) \nonumber \\
\nonumber & + \mathfrak{g}_0^{-1} * \mathfrak{g}_0^{-1} * \nabtil \mathfrak{g}_1 * \nabtil (\mathfrak{g}_1 - \mathfrak{g}_0) 
+ \mathfrak{g}_0^{-1}*\mathfrak{g}_0^{-1}*\nabtil \mathfrak{g}_0 * \nabtil(\mathfrak{g}_1 - \mathfrak{g}_0) \\
\nonumber&+ \mathfrak{g}_1^{-1}*\mathfrak{g}_0^{-1}*\mathfrak{g}_0^{-1} * \nabtil \mathfrak{g}_1 * \nabtil{\mathfrak{g}}_1 * (\mathfrak{g}_1 - \mathfrak{g}_0)\\
&+ \mathfrak{g}_0^{-1}*\tilde{g}^{-1} *\widetilde{Rm}*(\mathfrak{g}_1-\mathfrak{g}_0)+\mathfrak{g}_1^{-1}*\mathfrak{g}_0^{-1}*\mathfrak{g}_1*\tilde{g}^{-1}*\widetilde{Rm}*(\mathfrak{g}_1-\mathfrak{g}_0)\\
\nonumber&+ \mathfrak{g}_1^{-1} * \mathfrak{g}_1^{-1}* \mathfrak{g}_0^{-1}*\nabtil \mathfrak{g}_1 *\nabtil \mathfrak{g}_1*(\mathfrak{g}_1-\mathfrak{g}_0) + \mathfrak{g}_1^{-1} * \mathfrak{g}_0^{-1} * \nabtil^2 \mathfrak{g}_0 * (\mathfrak{g}_0 - \mathfrak{g}_1) \\
\nonumber &- 2(n-1) ( \mathfrak{g}_1 - \mathfrak{g}_0).
\end{align}
The tildes refer to quantities associated to $\gtil = g_0$, and the star * indicates contractions with respect to $g_0$.  The coefficients of these terms involve 
at most one derivative of $\mathfrak{g}_i$, $i=0,1$.  We have
\begin{align} \label{eqn:reg}
\| \mathfrak{g}_1(t) \|_{\holder^{k,\alpha}(M)} &\leq \| \mathfrak{g}_0(t) \|_{\holder^{k,\alpha}(M)} + \| \mathfrak{g}_1(t) - 
\mathfrak{g}_0(t) \|_{\holder^{k,\alpha}(M)} \nonumber \\ & \leq \| \mathfrak{g}_0(t) \|_{\holder^{k,\alpha}(M)} + C^{DT}_{\mathrm{fin}} \|g_1 - g_0 \|_{\holder^{k,\alpha}(M)}.
\end{align} 
Hence there is a $\holder^{k-1,\alpha}$ bound on the coefficients of \eqref{diffmetrics} depending only on $g_0(t), \mathfrak{g}_0(t), C_{\mathrm{fin}}$ and 
the initial distance $g_1 - g_0$.  We now invoke a more or less standard parabolic regularity result, which implies in this setting that
\begin{equation} 
 \| \mathfrak{g}_1(t) - \mathfrak{g}_0(t) \|_{\holder^{k+1,\alpha}} < {C^{\prime}} \| g_1 - g_0 \|_{\holder^{k,\alpha}}.
 \end{equation}
This requires some explanation.  First note that we really only need such an estimate in the interval $\tau/2 \leq t \leq \tau$; furthermore,
by \cite{BamlerKleiner}, it is known that there is an absolute lower bound on the time interval $\tau$ which depends only on estimates
for the initial geometry.  This means that although the constant $C^\prime$ will depend on $\tau$, this is still a uniform a priori estimate. 
Finally, we recall a linear estimate which asserts that if $H$ is the solution operator for a second order uniformly parabolic operator $L$ (where
uniformly parabolic corresponds to the geometric assumptions in our setting), so that 
\[
H(t, \cdot):  \holder^{k,\alpha}(M) \to \holder^{k,\alpha}(M)
\]
is bounded for every $t \geq 0$, then
\[
(tL) \circ H(t,\cdot) : \holder^{k,\alpha}(M) \to \holder^{k,\alpha}(M)
\]
is also bounded.  In particular, by interpolation, for $t \in [\tau/2, \tau]$, the norm of 
\[
H(t,\cdot):  \holder^{k,\alpha}(M) \to \holder^{k+2,\alpha}(M)
\]
is controlled in terms of the previous bound and $1/\tau$.  This is sufficient for our purposes.

Now return to the main line of reasoning. Iterate this argument.  Replacing $k$ by $k+1$ in \eqref{eqn:reg} gives
\begin{align*} 
\| \mathfrak{g}_1(t) \|_{\holder^{k+1,\alpha}(M)} &\leq \| \mathfrak{g}_0(t) \|_{\holder^{k+1,\alpha}(M)} + \| \mathfrak{g}_1(t) - \mathfrak{g}_0(t) \|_{\holder^{k+1,\alpha}(M)} \nonumber \\
& \leq \| \mathfrak{g}_0(t) \|_{\holder^{k+1,\alpha}(M)} + C^{\prime\prime} \|g_1 - g_0 \|_{\holder^{k,\alpha}(M)};
\end{align*} 
whence by parabolic regularity again, 
\begin{equation*} 
\| \mathfrak{g}_1(t) - \mathfrak{g}_0(t) \|_{\holder^{k+2,\alpha}} < C^{\prime\prime\prime} \| g_1 - g_0 \|_{\holder^{k,\alpha}}.
 \end{equation*}
Finally, estimating the DeTurck vector field and the diffeomorphism as in \cite{BGI}, we obtain
\[ 
\|g_1(t) - g_0(t) \|_{ \holder^{k,\alpha} } \leq C_{\mathrm{fin}} \| g_1 - g_0 \|_{\holder^{k,\alpha}}, \ \ \frac{\tau}{2} \leq t \leq \tau,
\]
for a new constant $C_{\mathrm{fin}}$.  This completes the proof.
\end{proof}

The proof of Theorem A in \cite{BGI} implicitly assumes the existence of an associated Ricci-DeTurck flow on the same interval as the Ricci flow. 
We justify this in the next section, using a short-time existence result combined with controlled time of existence to obtain a complete proof of 
that Theorem and to extend it to the setting of bounded geometry.

\subsection{Finite-time continuous dependence for the Ricci flow}
\label{sec:ctsdep-rest}
We now turn to the second step, which completes the proof of Theorem \ref{thm:ctsdep}.  
\begin{proof} Given $\tau$ as in the hypothesis Theorem \ref{thm:ctsdep}, choose $T \in (\tau,T_{max}(g_0))$. There exists a complete, bounded curvature Ricci flow solution $g_0(t)$ emanating from $g_0$ on $[0,T]$ which satisfies
	\[ 
	\sup_{ (p,t) \in M \times [0,T] } |\Rm(g_0(t))|_{g_0(t)} \leq K. 
	\]
	All higher covariant derivatives of $\Rm(g_0(t))$ are then bounded in terms of $K$, $\dim M$ and $t$ by Shi's estimates.
	
	We construct a sequence of harmonic map heat flows with source $(M,g_0(t))$ for $t$ in fixed time intervals $[t_{i-1},t_i]$, to be specified below, and with target $(M,g_0(t_{i-1}))$.
	Since the source and target metrics are both complete and enjoy the same curvature bounds, a result of Bamler and Kleiner \cite[Proposition A.9]{BamlerKleiner} 
	gives a time $T_1 > 0$,
	depending only on the curvature bounds of $g_0(t)$ on $[0,T]$ and $\dim M$, so that for any $a < b$ with $|b-a| < T_1$, the harmonic map heat flow 
	emanating from the identity map with source $(M,g(t))$, $t \in [a,b]$ and target $(M,g(a))$  exists and remains a diffeomorphism on $[a,b]$.  Thus 
	if we partition $[0,\tau]$ into $N$ subintervals $I_i := \{ [t_{i-1},t_i] \}_{i=1}^{N}$ with $t_i - t_{i-1} < T_1$, then on each $I_i$ there exist both a Ricci flow and a Ricci-DeTurck flow starting at $g(t_{i-1})$. We emphasize that the Ricci-DeTurck flow differs on each interval due to the choice of reference metric.
	
	By Proposition \ref{prop:ctsdep-rfrdf}, there exist sequences of constants $r_{i-1},
	C_{i} > 0$, so that if for a metric $g_1$ we have $$\| g_1 - g_0(t_{i-1}) \|_{\holder^{k,\alpha}} < r_{i-1},$$ then at the final time $t=t_i$ on the subinterval $I_i$, we have
	$$\|g_1(t_i) - g_0(t_i)\|_{\holder^{k,\alpha}} < C_{i} \|g_1 - g_0(t_{i-1})\|_{\holder^{k,\alpha}},$$ where as usual $g_1(t)$ is the unique solution of the Ricci flow starting at $g_1$.  It is convenient to assume that each $C_i \ge 1.$ 
	
	We may construct $g_1(t)$ on the entire time interval $[0,\tau]$ as follows. By definition of the constants $r_{0} , C_{1}$, if $$\|g_1-g_0\|_{\holder^{k,\alpha}} < r_{0},$$ then the unique Ricci flow solution starting at $g_1$ exists and satisfies  $$\|g_1(t_1) - g_0(t_1)\|_{\holder^{k,\alpha}} < C_{1}\|g_1 - g_0\|_{\holder^{k,\alpha}}.$$ We now shrink $r_{0}> 0$ if necessary so that if $\|g_1 - g_0\|_{\holder^{k,\alpha}}< r_{0}$, then $$C_{1}\|g_1 - g_0\|_{\holder^{k,\alpha}} < r_{1},$$ and so $\|g_1(t_1) - g_0(t_1)\|_{\holder^{k,\alpha}} < r_1.$ Then the solution $g_1(t)$ extends uniquely to the time interval $[t_1, t_2]$,  and satisfies 
	\begin{align*}
	\|g_1(t_2) - g_0(t_2)\|_{\holder^{k,\alpha}} &< C_{2} \|g_1(t_1) - g_0(t_1)\|_{\holder^{k,\alpha}} \\
	&< C_{2} C_{1}\|g_1 - g_0\|_{\holder^{k,\alpha}}
	\end{align*}
	Here we've used $C_{2} \ge 1$.
	In the next step we shrink $r_{0}$ if necessary so that $C_{2} C_{1} \|g_1-g_0|\|_{\holder^{k,\alpha}}< r_{2}$, and proceed as above. Continuing this a finite number of times, we have generated a solution $g_1(t)$ of the Ricci flow starting at $g_1$ that satisfies
	$$\|g_1(T) - g_0(T)\|_{\holder^{k,\alpha}} < \left( \prod_{i=1}^N C_{i} \right) ||g_1 - g_0\|_{\holder^{k,\alpha}}$$
	on the time interval $[0,T].$ Setting $C_{\mathrm{init}}$ equal to the final $r_{0}$ obtained in this process and $C_{\mathrm{fin}} = \prod_{i=1}^N C_i$ completes the proof. 
	
\end{proof}

\section{An application to Ricci flow near rotationally symmetric AH metrics}
\label{sec:appl}

As an application of the previous result, we prove the convergence stability for Ricci flow around asymptotically hyperbolic rotationally symmetric solutions which converge to the hyperbolic metric. This is {\it not} a small data result since we allow initial starting metrics 
which may be very far away from any fixed point of the flow.   This argument adapts directly to prove convergence stability in the
neighborhood of any infinite-time solution of Ricci flow (or any other geometric flow) which converges to a strictly stable stationary point
of the flow. We state this more formally below, see Theorem \ref{thm:gen}. 

We now consider the space of metrics on the open ball $B^n$ which are asymptotically hyperbolic (AH) in the sense that 
they are conformally compact, with sectional curvatures tending to $-1$ at infinity. Recall that a metric $g$ on $B^n$ (or indeed 
on the interior of any compact manifold with boundary) is defined to be conformally compact if, for any
defining function $\rho$ for $\del M$, the metric $\overline{g} = \rho^2 g$ extends to be a smooth metric on $\overline{B^n}$ (or
$\overline{M}$).  The prototype is the Poincar\'e metric $g_h = \rho^{-2} g_E$, where $g_E$ is the standard Euclidean metric
on the ball and $\rho = (1-|z|^2)/2$.   Any AH metric determines a conformal class, called its conformal infinity, on $\del M$.
An AH metric $g$ is called Poincar\'e-Einstein if it has this conformally compact AH structure
and in addition satisfies the Einstein equation $\Rc(g) = -(n-1)g$.   Any such metric is a fixed point of the normalized Ricci flow
\begin{equation} \label{eqn:NRF}
\frac{\partial g}{\partial t} = - 2 [ \, \Rc ( g(t) ) + (n-1) g(t) \, ].
\end{equation} 
We refer to this below simply as Ricci flow. 

As an aside, the perturbation result by Graham-Lee \cite{GL} and Biquard \cite{Biq} states that near to a given Poincar\'e-Einstein metric $g_0$
(which satisfies a generic nondegeneracy condition), there is an infinite dimensional family of metrics near to $g_0$ in the unweighted
H\"older norm $\holder^{k,\alpha}$, parametrized by their conformal infinities. These are all also stationary points of \eqref{eqn:NRF}. Hence any
stability result for normalized Ricci flow, e.g. for the hyperbolic metric $g_h$ on $B^n$, must require that the initial metric $g_0$ 
be asymptotic to $g_h$ at infinity. We thus assume that $g_0 - g_h$ lies in a weighted H\"older space
$\holder^{k,\alpha}_{\mu}(M) = \rho^{\mu} \holder^{k,\alpha}$ for some $\mu > 0$.  (These spaces can be defined easily for symmetric
$2$-tensors, and covariant derivatives and distances can then be taken with respect to $g_h$.)  Notice that any metric decaying
in this way to $g_h$ has the same conformal infinity. 

It has been proved by several authors \cite{{Bahuaud},{Bam}, {LiYin}, {QSW}, {SSS}} that $g_h$ is a stable fixed point of both the Ricci-DeTurck and
Ricci flows in spaces of metrics which are asymptotic at some rate to $g_h$. In particular, we have
\begin{theorem} \label{thm:stabhyp}  Let $(M=B^n,g_h)$ be the Poincar\'e ball, and fix $k\geq 2$ and $\mu \in (0,n-1)$. There exists $\vep > 0$ such 
that if  
\begin{equation}
\| 
g_0 -g_h 
\|_{\holder^{k,\alpha}_{\mu}} < \vep,
\end{equation}
then normalized Ricci-DeTurck flow starting at $g_0$ exists for all time and converges to $g_h$ in $C^{k,\alpha}_{\mu}$ norm.   Similarly, the Ricci flow solution 
starting at $g_0$ exists for all time and converges to $g_h$.
\end{theorem}

We briefly sketch a proof. Choose the reference metric $\gtil = g_h$.   By \cite{BGIM}, the linearized Ricci-DeTurck flow at any metric near to
$g_h$ in this norm is sectorial on $\holder^{k,\alpha}_{\mu}$.   For the hyperbolic metric itself, this linearization is $DF_{g_h} = \Delta_L^{g_h} - 2(n-1)$, 
where $\Delta_L^{g_h}$ is the Lichnerowicz Laplacian on symmetric $2$-tensors with respect to the hyperbolic metric.  

Acting on $L^2$, $DF_{g_h}$ is self-adjoint, and by the Koiso-Bochner formula \cite[pg.78]{LeeFredholm}, 
\[
\mathrm{spec}_{L^2} (DF_{g_h}) \subset \left[ (n-2), \infty \right).
\]
The spectrum of this operator on other function spaces can be determined by recognizing elements of the resolvent family $(\Delta_L^{g_h} - \lambda)^{-1}$ 
as $0$-pseudodifferential operators in the sense of \cite{Ma-edge}.   We sketch this briefly, but refer to Section 2.6 of \cite{BGIM} for more details.
One first defines the {\it indicial roots} of $\Delta_L^{g_h} - \lambda$ as the values $\gamma$ such that
\[
(\Delta_L^{g_h} - \lambda) \rho^{\gamma} (k/\rho^2) = \mathcal O( \rho^{\gamma+1}),
\]
for some two-tensor $k$ which is smooth up to $\del M$.  Note that the extra factor of $\rho^{-2}$ is inserted so that $\rho^\gamma (k/\rho^2)$ decays
(in norm) at least like $\rho^\gamma$.  In order for $\gamma$ to be an indicial root, there must be a leading order cancellation in the series expansion of
the left hand side, and hence these indicial roots are roots of an algebraic equation. For this particular operator, and working only with AH metrics (with 
sectional curvatures all tending to $-1$), there are six such values \cite[Proposition 3]{MazPac}, and these are independent of $p \in \del M$ 
but depend (algebraically) on $\lambda$; they arise in pairs, as $\gamma_1^{\pm}, \gamma_2^\pm, \gamma_3^{\pm}$, symmetric around $(n-1)/2$.  
If $\lambda$ is real and $\lambda < (n-2)$, these roots are ordered:
\[
\gamma_3^-(\lambda) < \gamma_2^-(\lambda) < \gamma_1^-(\lambda) <  \frac{n-1}{2} < \gamma_1^+(\lambda) < \gamma_2^+(\lambda) < \gamma_3^+(\lambda).
\]
On the other hand, if $\lambda > n-2$, these may become complex conjugate pairs, with $\mathrm{Re}\, \gamma_j^\pm = (n-1)/2$.  This
threshold $(n-1)/2$ occurs because $\rho^{(n-1)/2} (k/\rho^2)$ is just on the border of lying in $L^2$ near $\rho=0$, which is what determines 
the $L^2$ continuous spectrum. The point of all this is that by \cite[Theorem 3.27]{Ma-edge}, see also \cite[Proposition 9]{BiquardMazzeo}, 
$(\Delta_L^{g_h} - \lambda)^{-1}$ is bounded on $\holder^{k,\alpha}_\mu$ if and only if $\gamma_1^+(\lambda) > (n-1)/2 + \mu$.  In addition,
we can ensure that this inequality holds by choosing $\lambda$ to lie in a left half-plane which includes the imaginary axis.  More precisely: 
\begin{prop}
\label{prop:bdresolvent}
Fix $\mu \in (0, n-1)$.  Then there exists some $\omega > 0$ such that for all $\lambda$ with $\mathrm{Re} \lambda < \omega$, $\gamma_1^+(\lambda) > 
(n-1)/2 + \mu$, and hence the resolvent $R(\lambda)$ is bounded on $\holder^{k,\alpha}_\mu(M, g_h)$.  Consequently, $\{\lambda: \mathrm{Re}\, \lambda < \omega\}$
is contained in the $\holder^{k,\alpha}_\mu$ resolvent set. 
\end{prop}
From this we obtain, using the formalism of analytic semigroups (see e.g. \cite{Lunardi}), the stability result
\begin{corollary} 
\label{cor:LunStab}
The hyperbolic metric $g_h$ is a stable fixed point for the Ricci-DeTurck flow on $\holder^{k,\alpha}_\mu$. In other words, there exists $\vep > 0$ so that
if $g_0$ is AH and $||g_0 -g_h||_{\holder^{k,\alpha}_{\mu}} < \vep$, then the (normalized) Ricci-DeTurck flow starting at $g_0$ exists for all time and converges
in norm at an exponential rate to $g_h$: 
\begin{equation}
\| \mathfrak{g}_0 (t) - g_h \|_{\holder^{k,\alpha}_{\mu}} \leq C e^{-\omega t} \| g_0 - g_h \|_{\holder^{k,\alpha}_{\mu}}.
\label{expest}
\end{equation}
\end{corollary}
\begin{remark}
This same proof can be applied more generally. If $g$ is any Poincar\'e-Einstein metric, then the essential spectrum of $\Delta_L^g$ lies in
$[ n-2, \infty)$, but there could be point spectrum in $(-\infty, n-2)$. In many cases, it can be shown that this point spectrum is contained entirely
in the positive real line $(0, n-2)$, and if this is the case, then the analogue of Proposition \ref{prop:bdresolvent} and Corollary \ref{cor:LunStab} apply to $(M,g)$ as well. 
\end{remark}

We next consider the situation for the ungauged (normalized) Ricci flow.  First note that \eqref{expest} implies that the DeTurck vector field $W$ associated
to the solution $\mathfrak g_0(t)$ decays exponentially as $t \to \infty$.  It also decays at spatial infinity.  We can thus integrate this time-dependent 
vector field, via \eqref{eqn:ODEW}, and conclude that the one-parameter family of diffeomorphisms $\phi_t$ converges to a limiting diffeomorphism $\phi_{\infty}$. 
This implies that $g_0(t) = (\phi^{-1}_t)^* \mathfrak{g}_0(t)$ is the (unique) Ricci flow solution which emanates from $g_0$, exists for all time 
and converges to $\phi^*_{\infty} g_h$. 

We come finally to the main result of this section, namely that convergence stability holds for the Ricci flow beginning at any metric sufficiently close to a 
rotationally symmetric metric AH metric whose Ricci flow trajectory converges to $g_h$.  Suppose that $g$ is a smooth, rotationally symmetric AH metric on 
$B^n$ such that $g-g_h \in \holder^{k,\alpha}_{\mu}(M)$ for
some $k \geq 2$ and $\mu \in (0, n-1)$, but where this norm is only assumed to be finite, but not small.   The main theorem of \cite{BW} asserts that 
if the `tangential' sectional curvatures of $g$, i.e., the sectional curvatures associated to $2$-planes tangent to the angular $\mathrm{SO}(n)$ orbits, 
are strictly negative, then the Ricci flow $g(t)$ emanating from $g$ exists for all time, remains AH and rotationally symmetric, and converges exponentially
to $g_h$.  The novelty, of course, is that $g = g(0)$ may be quite far from $g_h$. 

To verify that there are AH rotationally symmetric metrics which lie outside any neighborhood of the hyperbolic metric, we consider metrics on a ball in $\mathbb R^n$ (or possibly all of $\mathbb R^n$) of the form 
\[ 
g = \phi^2(r) dr^2 + \psi^2(r) g_{\mathbb{S}^{n-1}}. 
\]
The hyperbolic metric assumes this form in several different ways. For example, in geodesic polar coordinates, 
\begin{equation}
g_h = dr^2 + \sinh^2r  g_{\mathbb{S}^{n-1}},
\end{equation}
while in polar coordinates in the Poincar\'e ball model,
\[
g_0 = \frac{ 4}{(1-r^2)^2} dr^2 + \frac{4r^2}{(1-r^2)^2} g_{\mathbb{S}^{n-1}}.
\]
We restrict here to rotationally symmetric metrics with $\psi(r) = \sinh r$ and make the ansatz that $\phi(r) = \sqrt{ 1 + w(r)^2 }$. Smoothness of $g_0$ at the origin requires
that $w$ has an even Taylor expansion in $r$, with $w(0) = 0$. The decay of $g_0$ to $g_h$ at infinity is implied by requiring that the Taylor coefficient $w^{(j)}(r) = O(e^{-\mu r})$
for $j \leq k+1$. With these conditions, $g_0$ is AH.  The tangential sectional curvatures then equal
\begin{equation}
\sec_T(r) = \frac{1}{\psi^2} - \frac{\psi_r^2}{\phi^2 \psi^2}.
\end{equation}
By a short calculation, if we assume $w^2 \leq \sinh^2 r$, then $\sec_T(r) < 0$.   Hence by choosing $w$ appropriately, we obtain the existence of 
rotationally symmetric AH metrics with $||g_0 - g_h||_{\holder^{k,\alpha}_{\mu}}$ arbitrarily large: $g_0$ can be chosen to lie outside any stability neighborhood of $g_h$
in this norm. 

We are in position to state and prove our main convergent stability result: 
 
\begin{theorem}
\label{thm:AHConvStab} Let $M = B^n$ and $g_h$ the Poincar\'e hyperbolic metric. Suppose that $g_{\star}$ is a complete, rotationally symmetric, AH metric 
with $g_{\star} - g_h \in \holder^{k,\alpha}_{\mu}$ for some $\mu \ in \ (0,n-1)$.  Suppose too that $g_{\star}$ satisfies the restriction on tangential curvatures
described above. Then there exists a constant $\delta > 0$ such that if $g_0$ is any AH metric with 
\begin{equation}
\| g_0 - g_{\star} \|_{\holder^{k,\alpha}_{\mu}} < \delta,
\end{equation}
then the normalized Ricci flow starting at $g_0$ exists for all time and converges to $g_h$.
\end{theorem}
\begin{proof}
By Theorem \ref{thm:stabhyp} there exists $\vep > 0$ so that any solution to Ricci flow starting in the $\holder^{k,\alpha}_{\mu}$ ball of radius $\vep$ around
$g_h$ exists for all time and converges to $g_h$. 

Given $g_{\star}$ as in the statement of Theorem \ref{thm:AHConvStab}, there exists $T > 0$ so that the Ricci flow $g_{\star}(t)$ starting from $g_{\star}$ exists for all time and 
satisfies $\|g_{\star}(t) - g_h\|_{\holder^{k,\alpha}_{\mu}} < \vep / 2$ for $t > T/2$.  

By Theorem \ref{thm:ctsdep}, there is a $\delta > 0$ so that if $\| g_0 - g_{\star} \|_{\holder^{k,\alpha}_{\mu}} < \delta$, then the Ricci flow solution starting
	from $g_0$ exists for at least time $3 T /4$, and $\| g_0(3T/4) - g_{\star}(3T/4) \|_{\holder^{k,\alpha}_{\mu}} < \vep / 2$.  Thus 
	$\| g_0( 3 T/4) - g_h \|_{\holder^{k,\alpha}_{\mu}} < \vep$.  We conclude that the unique continuation of $g_0(t)$ as a solution exists for all $t > 0$
	and converges exponentially to $g_h$. 
\end{proof}

As promised at the beginning of this section, it should be apparent that the arguments here and in the previous section are quite general, and can be easily
adapted to prove the following
\begin{theorem} \label{thm:gen}
Suppose that $\del_t g = F(g)$ is any ungauged geometric flow which admits a gauging, i.e. the addition of a term $G(g)$ tangent to the
diffeomorphism orbit and such that $D( F + G)$ is an elliptic geometric operator which is admissible in the sense of \cite{BGIM}.  Suppose that
$g_h$ is any stationary point of the ungauged and gauged flows which is strictly stable.   Let $g_0$ be any metric which is the initial point of
a solution $g_0(t)$ to the ungauged flow.  Then for any metric $g_1$ sufficiently near to $g_0$ in an appropriate weighted H\"older norm,
the ungauged flow $g_1(t)$ also converges to $g_h$, and the entire trajectory $g_1(t)$ remains close to $g_0(t)$ for all $t$. 
\end{theorem}

\end{document}